\documentclass{elsarticle}              
\usepackage{amsmath}
\usepackage[english]{babel}
\usepackage{lmodern}
\usepackage[T1]{fontenc}
\usepackage{amssymb}
\usepackage{graphicx}
\usepackage{array}
\usepackage{color}
\usepackage{multirow}
\usepackage[colorinlistoftodos]{todonotes}
\usepackage{algorithm}
\usepackage[noend]{algpseudocode}
\usepackage{dsfont} 
\usepackage{marvosym}
\usepackage{appendix}
\usepackage{stmaryrd}  

\newtheorem{definition}{Definition}

\newtheorem{theorem}{Theorem}
\newtheorem{corollary}{Corollary}

\begin{document}
\begin{frontmatter}

\title{A global constraint for the capacitated single-item lot-sizing problem}

\author[inpg]{Grigori German  }
\ead{grigori.german@grenoble-inp.fr}

\author[inpg]{Hadrien Cambazard  }
\ead{hadrien.cambazard@grenoble-inp.fr}

\author[uca]{ Jean-Philippe Gayon}
\ead{jean-philippe.gayon@grenoble-inp.fr}

\author[inpg]{Bernard Penz  }
\ead{bernard.penz@grenoble-inp.fr}


\address[inpg]{Univ. Grenoble Alpes, CNRS, Grenoble INP, G-SCOP, 38000 Grenoble, France}
\address[uca]{LIMOS, Universite Clermont Auvergne, France}

\begin{abstract}
The goal of this paper is to set a constraint programming  framework to  solve lot-sizing problems. More specifically, we consider a single-item lot-sizing problem with time-varying lower and upper bounds for  production and inventory. The cost structure includes time-varying holding costs, unitary production costs and setup costs. We establish a new lower bound for this problem by using a subtle time decomposition. We formulate this NP-hard problem as a global constraint and show that bound consistency can be achieved in pseudo-polynomial time and when not including the costs, in polynomial time.  We develop filtering rules based on existing dynamic programming algorithms, exploiting the above mentioned time decomposition for difficult instances. In a numerical study, we compare several formulations of the problem: mixed integer linear programming, constraint programming and dynamic programming. We show that our global constraint is able to find solutions, unlike the decomposed constraint programming model and that constraint programming can be competitive, in particular when adding combinatorial side constraints.
\end{abstract}

\end{frontmatter}

\section{Introduction}

The field of production planning addresses numerous complex problems covered by operations research and combinatorial optimization. In particular, lot-sizing problems have been broadly studied. The core problem \cite{wagner1958} and several variants have been solved by Dynamic Programming (DP) in polynomial time. Other variants (e.g. time varying production capacity and setup costs, multi-echelon) are NP-hard and are most of the time dealt with Mixed Integer Linear Programming (MILP) formulations (see e.g. \cite{pochet2006production,baranyMIC1984}). 

State-of-the-art approaches for complex lot-sizing problems are currently based on polyhedral techniques such as cutting plane algorithms and can handle a large class of problems with side-constraints. Nonetheless theses techniques may eventually fail when facing combinatorial additional constraints. In this paper, we investigate alternative generic approaches based on combinatorial techniques and designed within the Constraint Programming (CP) framework. The rationale is that a lot of algorithmic results have been obtained on the fundamental problems in this field over the last sixty years. We propose to reuse them as filtering mechanisms and building blocks of a generic solver for lot-sizing. This paper is a first step in that direction: we introduce a new global constraint \textsc{LotSizing} embedding the single-item lot-sizing problem. \textsc{LotSizing} appears to be especially generic and suits well in the modeling of a great variety of lot-sizing problems. The problem being NP-hard, we prove several complexity results on achieving different consistency levels for the constraint. We use a time decomposition to propose a new lower bound for the single-item lot-sizing problem. This time decomposition combined with classical results, namely DP algorithms, enables us to derive interesting cost-based filtering algorithms for \textsc{LotSizing}.

\paragraph{The capacitated single-item lot-sizing problem}
In this paper, we focus on the following single-item  lot-sizing problem  -- denoted by $(L)$  -- which is used as a building block to tackle more complex lot-sizing problems. The objective is to plan the production of a single product over a finite horizon of $T$ periods $\llbracket 1, T \rrbracket$ in order to satisfy a demand $d_t$ at each period $t$, and to minimize the total cost. The (per unit) production cost at $t$ is $p_t$ and a setup cost $s_t$ is paid if at least one unit is produced at $t$. A holding cost $h_t$ is paid for each unit stored at the end of period $t$. Furthermore the production (resp. the inventory) is bounded by minimal and maximal capacities $\underline{\alpha_t}$ and $\overline{\alpha_t}$ (resp. $\underline{\beta_t}$ and $\overline{\beta_t}$) at each period $t$.

Figure \ref{ULS} shows the problem as a graph with the variables and parameters on each arc. For each period, the incoming arcs corresponds to the possible production (vertical arcs) and inventory from the previous period (horizontal arcs). The outgoing arcs correspond to the demand (vertical arcs) and inventory at the end of the period (horizontal arcs). 
\begin{figure}[H]
\centering
   \includegraphics[scale = 0.55]{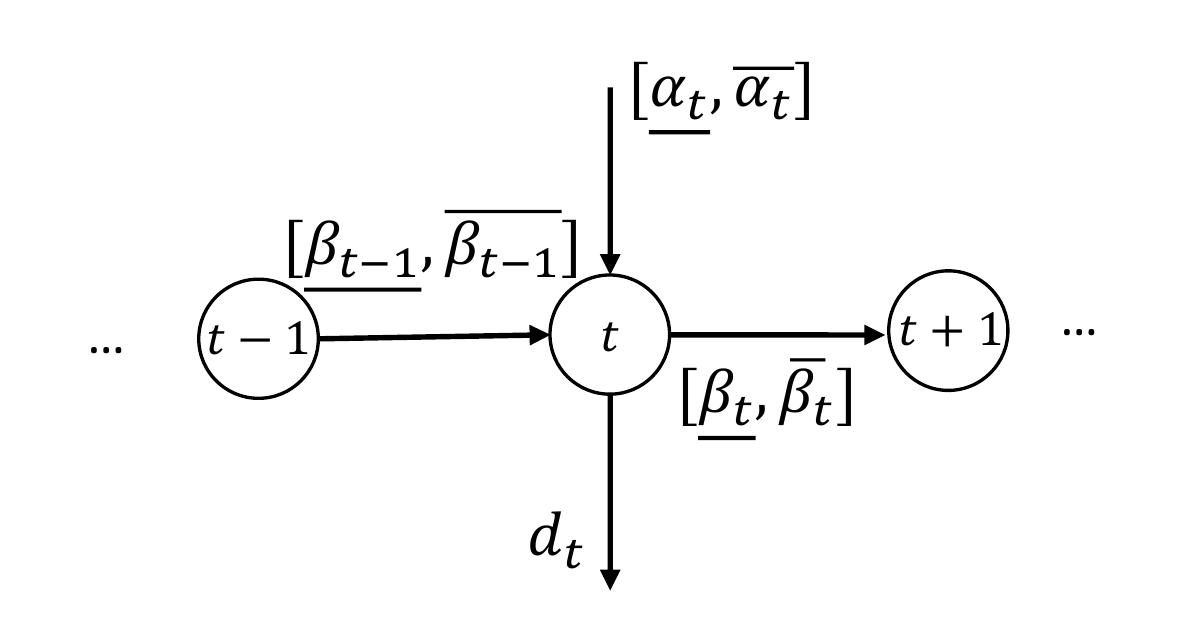}
   \caption{\label{ULS} Flow representation of the single-item lot-sizing problem}
\end{figure}

In the literature, one can find several models with upper bounds on either production or inventory. It is however unusual to include lower bounds. We make this assumption to be consistent with the CP framework that states domains for the variables.

\paragraph{Literature review}

The CP literature is very limited in the field of lot-sizing problems. To the best of our knowledge, \cite{stockingcost2014} is the only paper to study a lot-sizing-related global constraint. They consider a production planning problem in which a set of items has to be produced before their production deadline on a limited capacity machine, with the objective of minimizing stocking costs. This problem can be solved in polynomial time and is a special case of $(L)$ where production costs are set to zero ($p_t = s_t = 0, h_t=1$), the production and inventory lower bounds are set to zero ($\underline{\alpha_t} =\underline{\beta_t}=0$), the production upper bound is constant and there is no inventory upper bound ($\overline{\alpha_t} = \alpha, \overline{\beta_t}=+\infty$). It can be seen as a scheduling problem with deadlines and the objective of minimizing the total earliness ($P|\tilde{d}_j, p_j = 1|\sum E_j$ with Graham notation). In their approach, a decision variable is associated to each item and specifies in which period the item has to be produced.  This approach is suitable to deal with scheduling problems but seems less relevant to address lot-sizing problems for which large quantities of the same item can be produced in the same period. In \cite{houndji2019item}, the authors extend their global constraint to varying production capacities and stocking costs. Note that CP solvers have been used in the past to solve lot-sizing problems (see e.g. \cite{tarim2004echelon} for a distribution multi-echelon system).

We now focus the literature review on some special cases of problem $(L)$. There is no paper, to our knowledge, that considers lower bounds on both production and inventory -- see \cite{loveBI1973,van2008four} that consider inventory lower bounds only.
\cite{wagner1958} shows that the uncapacitated problem($\overline{\alpha_t}=\overline{\beta_t}=+\infty$) can be solved by DP in $O(T^2)$. This complexity has later been improved to $O(T \log T)$~\cite{federgruen1991simple,wagelmans1992economic,aggarwal1993improved}. When adding a constant production capacity and a constant setup cost, ($s_t = s$, $\overline{\alpha_t} = \alpha$), the problem can be solved in $O(T^4)$ with concave  costs \cite{florian1971deterministic} and in $O(T^3)$ with linear  costs  \cite{van1996t}. When the production capacity varies with time, the problem is NP-hard \cite{bitranCLSP1982}. Note that when $p_t = h_t = 0$, $(L)$ is equivalent to a knapsack problem. With time-varying inventory capacities, the problem can be solved in $O(T^2)$ with production and inventory setup costs \cite{loveBI1973,atamturk2008n2}. Finally \cite{atamturk2004study} is likely the only theoretical paper that studied the single-item problem with general capacities.

The rest of the paper is organized as follows. Section \ref{lsDescription} presents algorithms from the literature that will be re-used later. Section \ref{wisp} presents a new lower bound for this problem based on a time decomposition. Section \ref{LSglobalCst} presents the \textsc{LotSizing} global constraint and states complexity results for achieving bound and range consistency. Section \ref{filteringLS} presents cost-based filtering mechanisms for \textsc{LotSizing}. Section \ref{numResultsLS} compares numerically the performances of \textsc{LotSizing} with two MILP formulations, DP and a basic CP model. Section \ref{numResultsAC} considers two extensions with side constraints.

\section{Preliminaries}
\label{lsDescription}

This section presents classical MILP formulations and DP approaches, that will be used later in the paper. We also show that problem $(L)$ is equivalent to a problem without lower bounds on production and inventory.

\subsection{MILP formulations}
We list below a summary of the main notations.

\vspace{+5px}
\noindent\textbf{Parameters}
\vspace{-5px}
\begin{itemize}
	\item $T \in \mathbb{N}$: Number of periods.
	\item $p_t \in \mathbb{N}$: Unit production cost at $t$.
	\item $h_t \in \mathbb{N}$: Unit holding cost at $t$ (applied to the ending inventory).
	\item $s_t \in \mathbb{N}$: Setup cost at $t$ (paid if at least one item is produced at $t$).
	\item $d_t \in \mathbb{N}$: Demand at $t$.
	\item $\underline{\alpha_t}, \overline{\alpha_t} \in \mathbb{N}$: Minimal and maximal production quantities at $t$.
	\item $\underline{\beta_t}, \overline{\beta_t} \in \mathbb{N}$: Minimal and maximal inventory at the end of period $t$. 
	\item $I_0  \in \mathbb{N}$: Initial inventory.
\end{itemize}

\noindent \textbf{Variables}
\vspace{-5px}
\begin{itemize}
	\item $X_t \in \mathbb{N}$: Quantity produced at $t$.
	\item $Y_t \in \{0, 1\}$: Setup variable that equals $1$ if at least one item is produced at $t$.
	\item $I_t \in \mathbb{N}$: Inventory at the end of period $t$.
	\item $C \in \mathbb{N}$: Total cost.
	\item $\mathit{Cp} \in \mathbb{N}$: Sum of production costs.
	\item $\mathit{Cs} \in \mathbb{N}$: Sum of setup costs.
	\item $\mathit{Ch} \in \mathbb{N}$: Sum of holding costs.
\end{itemize}
We assume that parameters and domains are integers. We denote by $X$, $I$ and $Y$ the vectors $\langle X_1, \ldots, X_T \rangle$, $\langle I_1, \ldots, I_T \rangle$ and $\langle Y_1, \ldots, Y_T \rangle$. Without loss of generality we consider $I_0=0$. We also consider $I_T=0$. Indeed, we can compute the minimum mandatory quantity to store at the end of period $T$ from the production and inventory capacity constraints. If this quantity $q$ is strictly positive, we add a dummy period $T+1$ at the end of the time horizon with $p_{T+1} = h_{T+1} = 0$, $\overline{\alpha_{T+1}} = \overline{\beta_{T+1}} = 0$ and $d_{T+1} = q$.

\label{MILP}
Problem $(L)$ can be formulated as an aggregated MILP model (see e.g. \cite{pochet2006production}):
\begin{align}
&& \textit{minimize} \: \: C &= \mathit{Cp} + \mathit{Ch} + \mathit{Cs} \label{objMILP1}\\
&&I_{t-1} + X_{t} &= d_{t} + I_{t} && \forall \: t = 1 \ldots T \label{flotMILP1}\\
&&X_{t} &\leq \overline{\alpha_t} Y_t  && \forall \: t = 1 \ldots T \label{setupMILP1}\\
&&\mathit{Cp} &= \sum_{t=1}^{T} {p_{t}X_{t}} \label{CpMILP1}\\
(MILP\_AGG)&&\mathit{Ch} &= \sum_{t=1}^{T} {h_{t}I_{t}} \label{ChMILP1}\\
&&\mathit{Cs} &= \sum_{t=1}^{T} {s_{t}Y_t} \label{CsMILP1}\\
&&X_t &\in \{\underline{\alpha_t}, \ldots, \overline{\alpha_t}\} && \forall \: t = 1 \ldots T \label{domXMILP1}\\
&&I_t &\in \{\underline{\beta_t}, \ldots, \overline{\beta_t}\} && \forall \: t = 1 \ldots T \label{domIMILP1}\\
&&Y_t &\in \{0,1\} && \forall \: t = 1 \ldots T \label{domYMILP1}
\end{align}
where \eqref{flotMILP1} are the flow balance constraints for each period and \eqref{setupMILP1} are the setup constraints enforcing $Y_t$ to $1$ if a production is made at $t$. 
Finally, \eqref{CpMILP1}, \eqref{ChMILP1} and \eqref{CsMILP1} express the various costs. When $(L)$ is solved as a MILP, the variables $X$ and $I$ can be relaxed and considered real \cite{pochet2006production}.


$(L)$ can also be modeled as a facility location problem \cite{krarup1977plant}: the variables $I$ and $X$ are channeled to the variables $X_{tr}$, where $X_{tr}$ represents the proportion of demand $d_r$ produced in period $t$ and stored from $t$ to $r$. The model can be written as follows:
\begin{align}
&&\eqref{objMILP1}, \eqref{CpMILP1}&, \eqref{ChMILP1}, \eqref{CsMILP1}, \eqref{domXMILP1}, \eqref{domIMILP1}, \eqref{domYMILP1} && \nonumber \\
&&X_t &= \sum_{r = t}^{T}{ d_{r} X_{tr}} && \forall \: t = 1 \ldots T \label{defXMILP2}\\
(MILP\_UFL) &&I_t &= \sum_{q = 1}^{t}{ \sum_{r = t+1}^{T}{d_{r} X_{qr}}} && \forall \: t = 1 \ldots T \label{defIMILP2}\\
&&X_{tr} &\leq Y_{t} && \forall \: t = 1 \ldots T,\: r = t \ldots T \label{setupMILP2}\\
&&\sum_{t = 1}^{r}{X_{tr}} &= 1 && \forall \: r = 1 \ldots T \label{satisDtMILP2}\\
&&X_{tr} &\in [0,1] && \forall \: t = 1 \ldots T,\: r = t \ldots T \label{domXMILP2}
\end{align}
Though the number of variables is increased, this model has the advantage to tighten the \textit{big M constraints} \eqref{setupMILP1} of the first formulation by stating constraints \eqref{setupMILP2} and is known to provide an at least as good lower bound as the linear relaxation. Note that the shortest path reformulation (and $(l,S)$ inequalities) would provide an equivalent lower bound to UFL \cite{pochet2006production}.


\subsection{Linear relaxation}
\label{RLeqFLot}
Solving the linear relaxation of MILP\_AGG (i.e. $Y_t \in [0,1], \forall \: t \in \llbracket 1, T \rrbracket$) is equivalent to a minimum cost network flow problem~\cite{ahuja1993network}. The graph of this flow is presented in Figure \ref{RLeqFLotFig}. On each arc, $(u, c)$ represents the capacity ($u$) and unitary cost ($c$) of the arc. The units flow from  source node $S$ to  sink node $W$. On each production arc ($S$, $t$), the capacity is the production capacity and the cost is $p'_t = p_t + \dfrac{s_t}{\overline{\alpha_t}}$. On each inventory arc ($t$, $t+1$), the capacity is the inventory capacity and the cost is $h_t$. Finally on each demand arc ($t$, $W$), there must be exactly $d_t$ units and the unitary cost is $0$.

\begin{figure}[!h]
\centering
   \includegraphics[scale = 0.4]{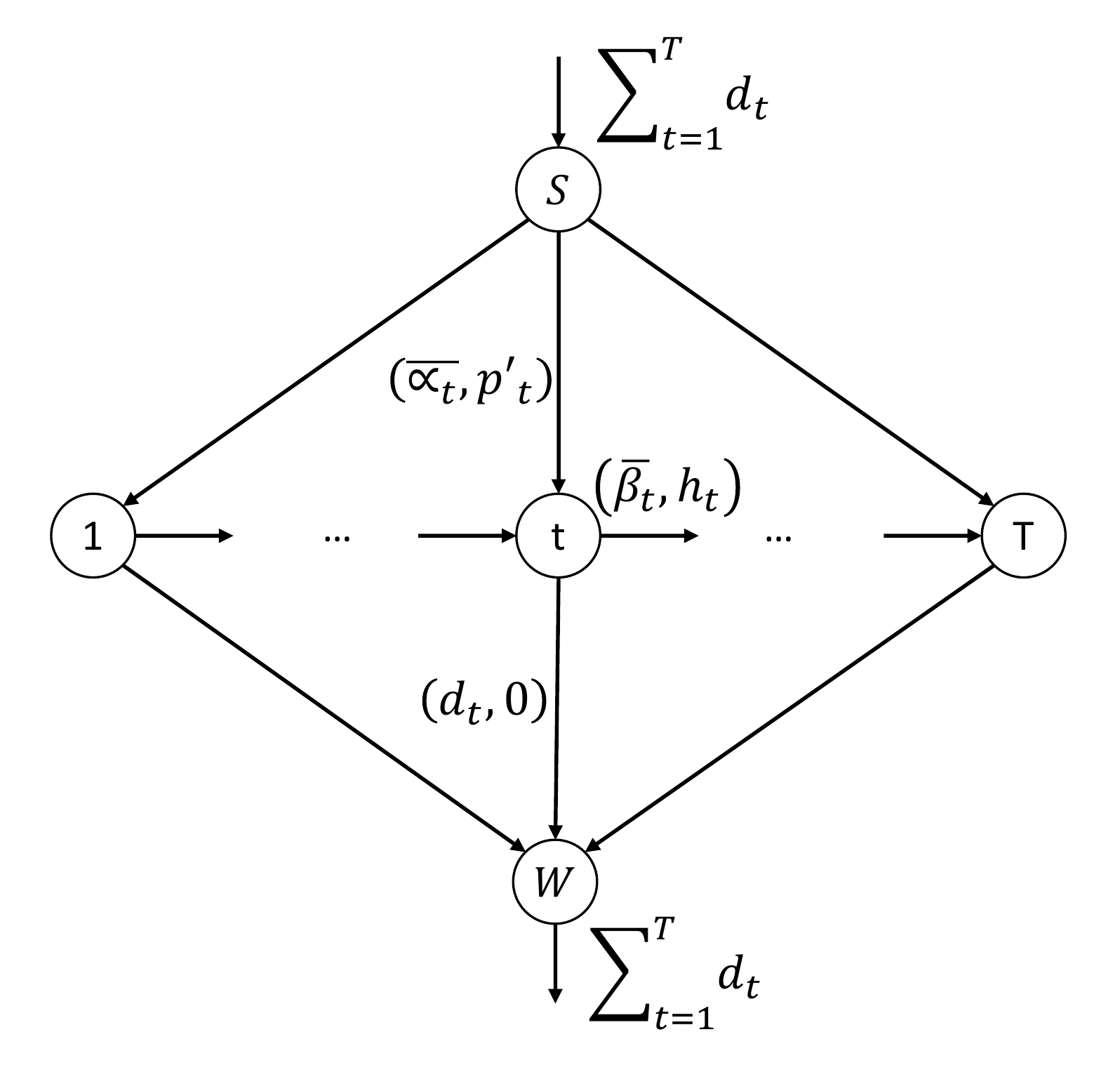}
   \caption{\label{RLeqFLotFig} The linear relaxation of MILP\_AGG is a minimum cost network flow problem}
\end{figure}

The flow problem can be solved in $O(T^2)$ with the successive shortest path algorithm \cite{ahuja1993network}.

\subsection{An equivalent problem without lower bounds }
\label{transformation}

In this subsection, we show that problem $(L)$ is equivalent to a problem without lower bounds. It will allow us to re-use several classical lot-sizing algorithms and will also simplify the presentation of some results.

Solving a maximum flow problem on a network with lower bounds on flows is equivalent to solving a maximum flow problem on a transformed network without lower bounds as shown in \cite{ahuja1993network}. We denote by $(L')$ the resulting problem of that transformation applied to $(L)$ slightly adapted to take into account setup costs. The parameters of $(L')$ are:
\begin{align*}
	& \underline{X'_t} = 0 \: \text{ and } \: \overline{X'_t} = \overline{X_t} - \underline{X_t} \\
	& \underline{I'_t} = 0 \: \text{ and } \: \overline{I'_t} = \overline{I_t} - \underline{I_t} \\
	& p'_t=p_t \\
	& h'_t=h_t \\
	& s'_t = \left\{
			\begin{array}{rl}
  				0  &\text{if}\: \underline{X_t}>0\\
				s_t  &\text{otherwise}
			\end{array}
		\right. \\
    & d'_t = d_t + \underline{I_{t}} - \underline{X_t} - \underline{I_{t-1}} 
\end{align*}

$(X,I)$ is a solution of $(L)$ if and only if $(X', I')$ is a solution of $(L')$. The intuition is that the production and inventory lower bounds are considered as mandatory quantities. As these quantities must be produced/stored at a precise period, no decisions have to be made about them and thus they can be removed from the problem. Note that the demands are also affected by the transformation.

The  mandatory costs associated to the lower bounds  are:
  \begin{align*}
{\mathit{Cp}_{min}} &= \sum_{t=1}^{T}{p_t \: \underline{X_t}}, \\
{\mathit{Ch}_{min}} &= \sum_{t=1}^{T}{h_t \: \underline{I_t}}, \\
{\mathit{Cs}_{min}} &= \sum_{t=1}^{T}{s_t \: \mathds{1}_{\underline{X_t} > 0}}, \\
{\mathit{C}_{min}} &= {\mathit{Cp}_{min}} + {\mathit{Ch}_{min}} + {\mathit{Cs}_{min}}.
  \end{align*}
and the variables of $(L)$ and $(L')$ are linked  as follows:
  \begin{align*}
{X_t} &= {X_t'} + \underline{X_t},\\
{I_t} &= {I_t'} + \underline{I_t},\\
\mathit{Cp} &=  \mathit{Cp'} + \mathit{Cp}_{min},\\
\mathit{Ch} &=  \mathit{Ch'} + \mathit{Ch}_{min},\\
\mathit{Cs} &=  \mathit{Cs'} + \mathit{Cs}_{min},\\
\mathit{C} &=  \mathit{C'} + \mathit{C}_{min}.
  \end{align*}

Note that if the final demand $d'_t$ is negative then $d_t+I_{t} < X_t + I_{t-1}$. This cannot be if constraints $\{(2), (3), (7), (8), (9)\}$ -- which correspond to the feasibility of (L) -- are bound consistent. We will show later in this paper (in section \ref{filteringLS}) that we can assume this property when removing the lower bounds.

\subsection{Dynamic programming}\label{DP}

$(L)$ can also be solved via DP \cite{florian1971deterministic}. We provide here the algorithm without lower bounds on production and inventory. The algorithm (called DPLS in the paper) iterates over the inventory levels. We denote $\textsl{f} \: (t, I_t)$ as the minimum cost for producing the demands from $d_1$ to $d_t$ knowing that the stock level at $t$ is $I_t$:
\begin{align}
&\forall \: t \in \llbracket 1, T\rrbracket \text{ and } \forall \: I_t \in \llbracket 0, \overline{\beta_t}\rrbracket \nonumber \\
&\textsl{f} \: (t, I_t) = \min_{I_{t-1} = a \ldots b} \: \{\textsl{f} \: (t-1, I_{t-1})+ \mathds{1}_{X_t > 0} s_t + p_t X_t + h_t I_t\} \label{DPeq}
\end{align}
where $a = \max{\{0, d_t + I_t - \overline{\alpha_t}\}}$, $b = \min{\{\overline{\beta_{t-1}},  d_t + I_t}\}$ and $X_t = I_t + d_t - I_{t-1}$. We define $I_{max} = \max \: \{\overline{\beta_t}\: | \: t \in \llbracket 1, T\rrbracket\}$. The initial states are $ \textsl{f} \: (0, 0) = 0  \text{ and } \forall \: I_t \in \llbracket 1, I_{max}\rrbracket \:, \: \textsl{f} \: (0, I_t) = + \infty$. The value $\textsl{f} \: (T, 0)$ gives the optimal cost of $(L)$. This dynamic programming algorithm runs in pseudo-polynomial time $O(T  I_{max}^2)$.

Note that DPLS consists in finding a shortest path in the graph for which there is a node for each inventory level at each period. The cost on an arc between two nodes $(t, I_t)$ and $(t+1, I_{t+1})$ corresponds to the cost for satisfying demand $d_t$ and having an inventory level $I_{t+1}$ at the end of period $t+1$ knowing that there was an inventory level $I_t$ at the end of period $t$.

DPLS is a DP algorithm referred to as "forward" since it considers the periods in chronological order. We can also write the reverse (or "backward") DPLS. Let $\textsl{f}_{r} (t, I_t)$ be the minimum cost for producing the demands from $d_{t+1}$ to $d_T$ knowing that the stock level at $t$ is $I_t$:
\begin{align*}
&\forall \: t \in \llbracket 0, T-1 \rrbracket \text{ and } \forall \: I_t \in \llbracket 0, \overline{\beta_t} \rrbracket \\
&\textsl{f}_{r} (t, I_t) = \min_{I_{t+1} = c \ldots d} \{ \textsl{f}_{r} (t+1, I_{t+1}) + \mathds{1}_{X_{t+1} > 0} s_{t+1} + p_{t+1} X_{t+1} + h_{t+1} I_{t+1}\}
\end{align*}
where $c = \max{\{0,I_t - d_{t+1}\}}$, $d =  \min{\{\overline{\beta_{t+1}},I_t - d_{t+1}+\overline{\alpha_{t+1}}\}}$ and $X_{t+1} = d_{t+1}+I_{t+1}-I_t$. The initial states are $\textsl{f}_{r} (T, 0) = 0$ and $\forall \: I_t \in \llbracket 1, I_{max} \rrbracket \:, \: \textsl{f}_{r}(T, I_t) = + \infty$. 
As described above, $\textsl{f}_{r} (t, I_t)$ can be seen as the shortest path from the node $(t, I_t)$ to the node $(T,0)$. The value $\textsl{f}_{r} \: (0, 0)$ gives the optimal cost of $(L)$.

\section{A new lower bound for the single-item lot-sizing} 
\label{wisp}

In this section, we present a new lower bound for the total cost $C$ and how it can be adapted for the setup cost $C_s$. The general idea is to decompose $(L)$ into sub-problems, then to compute a lower bound on each of these sub-problems and finally combine them at best to find a global lower bound. We suppose here that $\underline{\alpha_t} = \underline{\beta_t} = 0, \: \forall \: t \in \llbracket 1, T\rrbracket$. This assumption is not restrictive as production and inventory lower bounds can be easily removed in $(L)$ as shown in \ref{transformation}.

\subsection{Lot-sizing sub-problem}
A sub-problem $(L_{u,v})$, with $u < v$, is defined exactly as $(L)$ except that:

\begin{align*}
&d_t = 0 &\forall \: t &\notin \llbracket u, v\rrbracket\\
&s_t = 0 &\forall \: t &\notin \llbracket u, v\rrbracket
\end{align*}

The cost variables  of $(L_{u,v})$ are denoted $C^{uv}$, $C_{p}^{uv}$, $C_{h}^{uv}$ and $C_{s}^{uv}$, corresponding to the total cost, sum of production costs, sum of holding costs and sum of setup costs of $(L_{u,v})$. Figure \ref{SubProblem} illustrates the data used in sub-problem $(L_{u,v})$.
\begin{figure}[!h]
\centering
   \includegraphics[scale = 0.52]{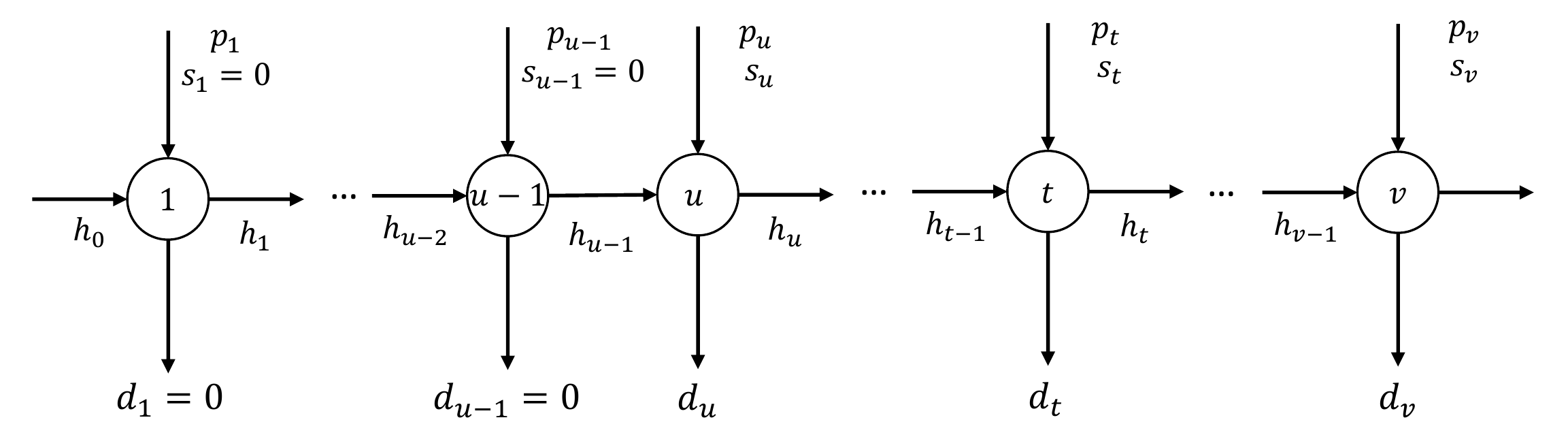}
   \caption{\label{SubProblem} Sub-problem $(L_{u,v})$}
\end{figure}

Sub-problem $(L_{1T})$ corresponds to the entire problem $(L)$.  As there is no demand after period $v$ and no lower bounds of production, any solution of $(L_{u,v})$ is dominated by a solution with null inventory at the end of period $v$ ($I_v = 0$). Note also that, in $(L_{u,v})$, some demands in $\{d_u, d_{u+1}, \ldots, d_v\}$ can be satisfied by a production made without setup cost before period $u$. Finally, an optimal solution of $(L_{u,v})$ provides a lower bound of the cost for satisfying the set of demands $\{d_u, d_{u+1}, \ldots, d_v\}$ in problem $(L)$. Indeed $(L_{u,v})$ is a relaxation of problem $(L)$. 


There are  $T(T-1)/2$ sub-problems and we order them by increasing end times first, then by increasing start times (see Table \ref{table:Index}). Sub-problem $(L_{u_i, v_i})$ will be referred to as sub-problem $i$. 
\begin{center}
\begin{table}[h!]
$$\begin{array}{l|c|c|c|c|c|c|ccc|c}
\text{Index} & 1 & 2 & 3 & 4 & 5 & 6 && \ldots && \frac{T(T-1)}{2}\\\hline
\text{Sub-problem} &(L_{1, 2})&(L_{1, 3})  & (L_{2, 3})  &(L_{1, 4})   & (L_{2, 4})  &(L_{3, 4}) && \ldots& &(L_{T-1, T}) 
\end{array} $$
\caption{Indexing sub-problems}\label{table:Index}
\end{table}
\end{center}

\begin{definition}
Sub-problems $(L_{u,v})$ and $(L_{u',v'})$ are \textbf{disjoint} if $$\llbracket u, v \rrbracket \cap \llbracket u', v'\rrbracket = \emptyset$$
\end{definition}

\subsection{Combining disjoint sub-problems provides a lower bound}
\label{WISP}
For sub-problem $i$, we denote by $w_i$ a lower bound of its total cost $C^{u_i v_i}$. Disjoint sub-problems can be combined to obtain a lower bound for the total cost of $(L)$.

\begin{theorem}
For any set $S$ of disjoint sub-problems, we have
$$\sum_{i \in S} {w_i}\leq C.$$
\end{theorem}

\begin{proof}
Let $E^*$ be an optimal production plan for $(L)$ of cost $C^*$ and $S$ be a set of disjoint sub-problems. We will build from $E^*$ a feasible solution to each sub-problem $i$ of $S$ and prove that their costs add up to less than $C^*$.

Consider a sub-problem $i$ in $S$. For each demand $d_t$, $t \in \llbracket u_i, v_i\rrbracket$ we produce $d_t$ at the same periods as it is produced in $E^*$ (with a First Come First Served policy).  We obtain by this process a feasible solution to sub-problem $i$ and denote its cost by $K_i$.

As the sub-problems in $S$ are disjoint and we keep the same production orders, the sum of setup costs paid in all of these sub-problems is less than or equal to the sum of setup costs paid in $E^*$. The production and inventory costs are identical to the costs paid in $E^*$ for the demands included in $\cup_{i \in S}\llbracket u_i, v_i\rrbracket$. It follows that $\sum_{i \in S} {K_i} \leq C^*$. Finally, as for each sub-problem $i$, $w_i$ is a lower bound of $C^{u_i v_i}$, we get: $\sum_{i \in S} {w_i} \leq \sum_{i \in S} {K_i}$.
\end{proof}

\subsection{Combining lower bounds at best}
Given a lower bound for each sub-problem, we wish to find the best lower bound of $C$, i.e. to determine the set $S$ of disjoint sub-problems that maximizes $\sum_{i \in S} {w_i}$. 

This problem can be seen as a Weighted Interval Scheduling Problem (WISP)  which can be solved in $O(n \log(n))$ where $n$ is the number of intervals \cite{kleinberg2006algorithm}. The algorithm  sorts the intervals in $O(n \log(n))$ and then applies a DP that runs in $O(n)$. In our case, there are $n=T(T-1)/2$ intervals (sub-problems) which are already sorted and the DP algorithm (called DPWisp in the paper) runs in $O(T^2)$. 

We use the indexing of intervals given in Table \ref{table:Index} and we denote by $\textsl{wisp}(i)$ the maximal weight that can be achieved using the $i$ first intervals. The forward DP writes as 
\begin{align*}
&\textsl{wisp}(0) = 0\\
&\textsl{wisp}(i) = \max \: \{\textsl{wisp}(i-1), \textsl{wisp}(prec_i) + w_i\}, \forall \: i = 1,\ldots, n
\end{align*}
where $prec_i$ is the biggest integer, smaller than $i$ ($prec_i < i$), such that the intervals $prec_i$ and $i$ are disjoint. For each sub-problem $i$ such that $u = 1$, we define $prec_i = 0$. Hence $prec_i$ is the first interval before the $i^{th}$ one that is disjoint with it. For instance, $[3,4]$ is the $6^{th}$ interval and $prec_{6} = 1$ since $[1,2]$ and $[3,4]$ are disjoint while $[2,3]$ and $[3,4]$ are not. The value $\textsl{wisp}(n)$ is then a lower bound of the global cost $C$.

We can also write the reverse version DPWisp (with the sub-problems considered backwards). The sub-problems are now sorted by decreasing start times first, then by decreasing end times. 
\begin{align*}
&\textsl{wisp}_r(n) = 0\\
&\textsl{wisp}_r(i) =  \max\{\textsl{wisp}_r(i+1), w_i + \textsl{wisp}_r(succ_i)\}, \forall \: i = 1,\ldots, n
\end{align*}
where $succ_i$ is the biggest integer, greater than $i$ ($succ_i>i$), such that the intervals $succ_i$ and
$i$ are disjoint.

\subsection{Computing lower bounds for sub-problems}
\label{ExamplesLB}
In our algorithms, we will  solve exactly the sub-problem by DP when the size is reasonable and solve the linear relaxation otherwise. The optimal cost of sub-problem $(L_{u,v})$ can be computed with DPLS (see Section \ref{DP}) applied to the periods $u$ to $v$. We need to pre-compute $\textsl{f} \: (u-1, q_{u-1})$ for $q_{u-1} \in \llbracket 0, \min \: \{\sum_{t=u}^{v}{d_t},\: \overline{\beta_{u-1}} \} \: \rrbracket $. This can be done with a greedy algorithm which determines the cheapest periods in order to produce the requested quantity and to store it until $u$. The DPLS applied to a sub-problem $(L_{u,v})$ runs then in $O((v - u + 1)(I_{max}^{uv})^2)$ where $I_{max}^{uv} = \max \{\overline{\beta_t} \: | \: t \in \llbracket u, v\rrbracket\}$.


The linear relaxation can be seen as a minimum cost flow problem (see Section \ref{RLeqFLot}) and can be solved in $O(T^2)$.




\subsection{Adaptation to a lower bound on setup costs}
\label{wispCs}
We can re-use this previous approach to obtain a lower bound on the setup cost variable $\mathit{Cs}$. Sub-problems are defined similarly except that we remove unitary  production costs and inventory costs ($p_t=h_t=0,\forall t$).

Note that we could use the same approach for $\mathit{Cp}$ and $\mathit{Ch}$. However, it is not necessary as $(L)$ is polynomial	 when removing setup costs.

\section{The lot-sizing global constraint}
\label{LSglobalCst}

In this section, we provide some CP background before presenting the \textsc{LotSizing} global constraint. We also study the complexity of achieving different consistency levels.

\subsection{Constraint programming background}
A Constraint Satisfaction Problem (CSP) \cite{HookerMS06} consists of a set of variables, with a finite domain of values for each variable, and a set of constraints on these variables. Upper cases are used for variables (e.g. $V_i$) and lower cases for values (e.g. $v_i$). We denote by $D(V_i)$ the domain of variable $V_i$ and by $\overline{V_i}$ (resp. $\underline{V_i}$)  the minimum (resp. maximum) value in $D(V_i)$. 

Let $c$ be a constraint on variables $\langle V_1, \ldots, V_n \rangle$. A \textit{support} for $c$ is a tuple  $\langle v_1,\ldots, v_n \rangle$ which satisfies $c$  and such that $v_i \in D(V_i)$ for each variable $V_i$.  A \textit{bound support} is a tuple  $\langle v_1,\ldots, v_n \rangle$ which satisfies $c$ and such that $\underline{V_i} \leq v_i \leq \overline{V_i}$ for each $V_i$.

A variable $V_i$ is \textit{arc consistent} (AC)  for constraint $c$ if each value of $D(V_i)$  belongs to a support for $c$. A variable $V_i$ is \textit{bound consistent} (BC) for constraint $c$ if $\underline{V_i}$ and $\overline{V_i}$ belong to a bound support for $c$. A variable $V_i$ is \textit{range consistent} (RC) for constraint $c$ if each value of $D(V_i)$ belongs to a bound support for $c$.  A constraint $c$ is  AC (resp. BC, RC) if all its variables are AC (resp. BC, RC). A CSP problem is  AC (resp. BC, RC) if each constraint is AC (resp. BC, RC).

The following example illustrates the three notions of AC, BC and RC. Consider the following linear constraint over two integer variables $x$ and $y$:
$$2x=y, \quad \quad x \in \{1,2,4\} \:\: \text{and} \:\: y \in \{ 4,5,6,7,8\}$$ 

\begin{table}[!h]
\footnotesize
\begin{center}
\begin{tabular}{|c|c|c|}
\hline
Consistency level & $D(x)$ & $D(y)$ \\
\hline
Initial domains & $\{1,2,4\}$ & $\{ 4,5,6,7,8\}$	\\
\hline
BC & $\{2,4\}$ & $\{ 4,5,6,7,8\}$	\\
\hline													
RC & $\{2,4\}$ & $\{4,6,8\}$	\\
\hline													
AC & $\{2,4\}$ & $\{ 4,8\}$	\\
\hline																							
\end{tabular}
\normalsize
\caption{\label{DomC} Three consistency levels}
\end{center}
\end{table}

The three levels of consistency are applied to the example and showed in table \ref{DomC}. The \textit{bound consistent} domains are $\{2,4\}$ and $\{ 4,5,6,7,8\}$ since we just check if the bounds belong to a bound support: only $1$ can be removed from the domain of $x$ as $\langle 1,2 \rangle$ is not a bound support. For instance, the bound support for $y=8$ is $\langle 4,8 \rangle$. The \textit{range consistent} domains are $\{2,4\}$ and $\{4,6,8\}$ since values $5$ and $7$ do not belong to a bound support. Indeed $\langle 2.5,5 \rangle$ and $\langle 3.5,7 \rangle$ are not bound supports as the values in a bound support must be in $\mathbb{Z}$. The value $6$ for $y$ is range consistent since $\langle 3,6 \rangle$ is a bound support. Finally the \textit{arc consistent} domains are $\{2,4\}$ and $\{ 4,8\}$ as we remove all inconsistent values.

\subsection{Definition}

We formally define here the global constraint \textsc{LotSizing}. This constraint is stated on the variable vectors $X = \langle X_1, \ldots, X_T \rangle$, $I = \langle I_1, \ldots, I_T \rangle$, $Y = \langle Y_1, \ldots, Y_T \rangle$ and the four cost variables $\mathit{Cp}$, $\mathit{Ch}$, $\mathit{Cs}$, $C$ of $(L)$. The data of the problem is denoted by $data = \{(p_t, h_t, s_t, d_t, \underline{\alpha_t}, \overline{\alpha_t}, \underline{\beta_t}, \overline{\beta_t}) \: | \: t \in \llbracket 1, T\rrbracket\}$.

\begin{definition}
$\textsc{LotSizing} (X, I, Y, \mathit{Cp}, \mathit{Ch}, \mathit{Cs}, C, data)$ has a solution if and only if there exists a production plan, solution of $(L)$ that satisfies:
\begin{align}
Cp &\leq \overline{Cp} \label{ubCp}\\
Ch &\leq \overline{Ch}\label{ubCh}\\
Cs &\leq \overline{Cs}\label{ubCs}\\
C &\leq \overline{C}\label{ubC}
\end{align}

\end{definition}

The $\textsc{LotSizing}$ global constraint has a solution if and only if the set of constraints $\{\eqref{flotMILP1}, \ldots, \eqref{domYMILP1}, \eqref{ubCp}, \ldots, \eqref{ubC}\}$ has a solution.

\subsection{Complexity}
\label{complexite}
This subsection presents theorems on the complexity of achieving BC, RC or AC on \textsc{LotSizing} and one of its restrictions that does not take into account the costs and focuses on the flow equations. Let us first give a property on the complexity of $(L)$.\\

\noindent \textbf{Property.}
\textit{Problem $(L)$ with $h_t=p_t=\underline{\beta_t}=\underline{\alpha_t} = 0$ and $\overline{\beta_t}=+\infty$ is NP-hard  \cite{florian1980deterministic}.}\\


\begin{theorem}
Achieving BC for \textsc{LotSizing} can be done in pseudo-polynomial time.
\end{theorem}

\begin{proof}
BC for \textsc{LotSizing} can be achieved by solving a Shortest Path Problem with Resource Constraints (SPPRC) in the graph of DPLS (see Section \ref{DP}) where the resources are the three intermediate costs of respective capacities $\overline{Cp}, \overline{Ch},\overline{Cs}$ and the global cost $C$ is the objective. Finding the shortest path in this graph while respecting the three resources at the final node gives a bound support regarding the three intermediate costs. SPPRC is known to be weakly NP-hard~\cite{feillet2004exact}.
\end{proof}

\

We denote by $(L_r)$ the feasibility problem associated to $(L)$.
\begin{align*}
 &&I_{t-1} + X_{t} &= d_{t} + I_{t} && \forall \: t = 1 \ldots T \\
&&X_{t} &\leq \overline{\alpha_t} Y_t  && \forall \: t = 1 \ldots T \\
(L_r) &&X_t &\in \{\underline{\alpha_t}, \ldots, \overline{\alpha_t}\} && \forall \: t = 1 \ldots T \\
&&I_t &\in \{\underline{\beta_t}, \ldots, \overline{\beta_t}\} && \forall \: t = 1 \ldots T \\
&&Y_t &\in \{0,1\} && \forall \: t = 1 \ldots T 
\end{align*}
The  constraints of $(L_r)$ describes the dynamics of the problem, without considering costs. 

\begin{theorem}
\label{thBCPoly}
	Achieving BC on $(L_r)$ can be done in $O(T)$.
\end{theorem}

\begin{proof}
Figure \ref{BergeAcyclic} represents the constraint network $(L_r)$ as well as the corresponding intersection graph. The rectangular-shaped constraints are the flow balance constraints and the dashed oval ones are the setup constraints. The intersection graph is built as follows: we set a vertex for each constraint and two vertices are linked if and only if the corresponding constraints have at least one variable in common. As the intersection graph is acyclic and each pair of constraints has at most one variable in common, the constraint network is Berge-acyclic. It is known that if we filter each constraint of a Berge-acyclic constraint network in an appropriate order then each constraint needs only to be woken twice in order to reach the fix-point~\cite{lagerkvist2009propagator}. Each constraint of the network can be filtered in $O(1)$, hence we can achieve BC on $(L_r)$ in $O(T)$.
\begin{figure}[!h]
\centering
   \includegraphics[scale = 0.48]{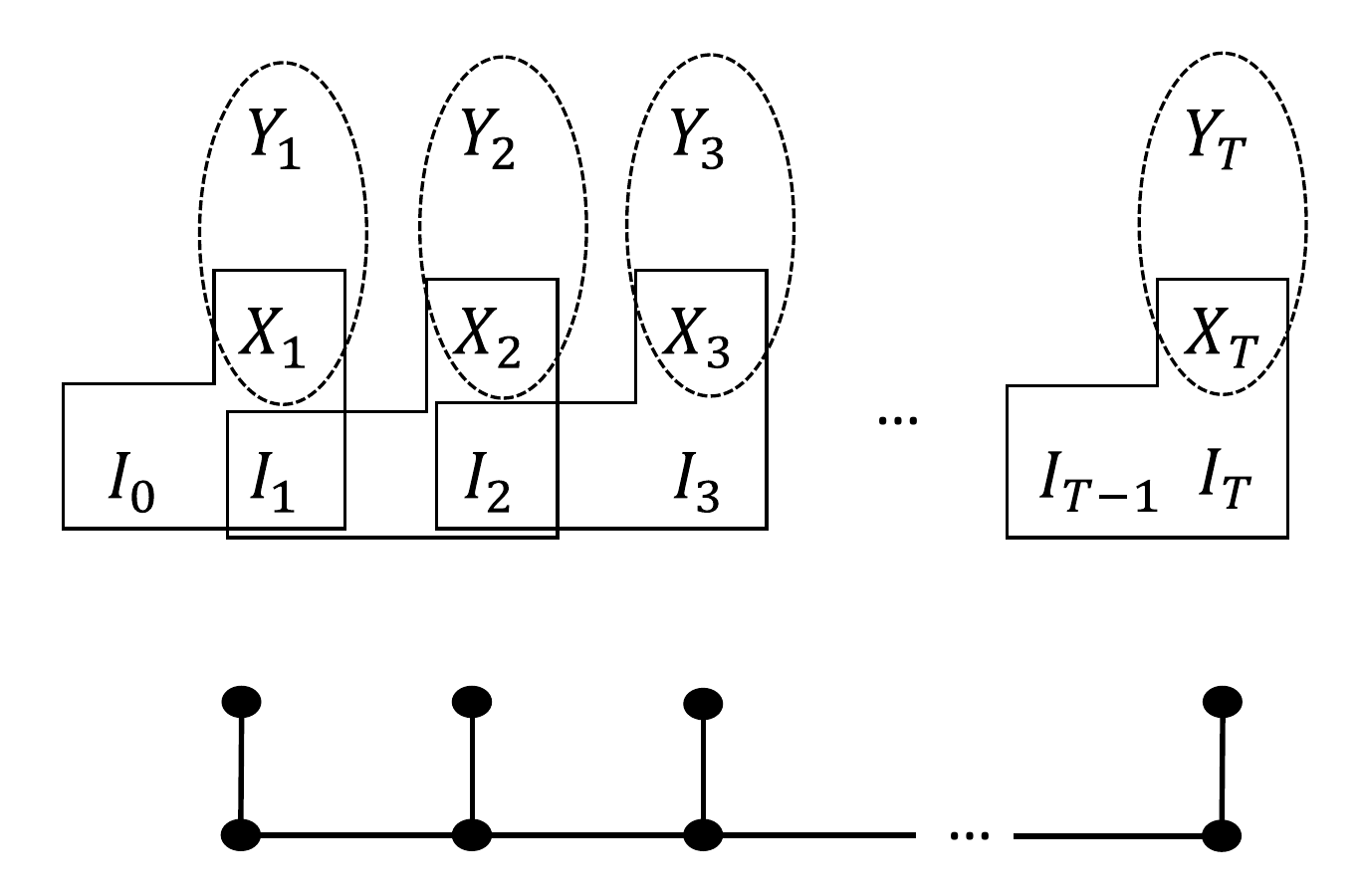}
   \caption{\label{BergeAcyclic} The constraint network $(L_r)$ and the corresponding intersection graph}
\end{figure}
\end{proof}

In order to investigate RC for $(L_r)$, we consider $(F)$ the more general problem of finding an integer flow in a directed graph $G=(V,E)$:
\begin{align}
&&l_{ij} \leq x_{ij} &\leq u_{ij} && \forall \:  (i,j) \in E \nonumber\\
(F) &&\sum_{j\in \delta^+_i}{x_{ij}} - \sum_{j\in \delta^-_i}{x_{ji}} &= b_i && \forall \: i \in V \nonumber\\
&&x_{ij} &\in \mathbb{N} && \forall \:  (i,j) \in E  \nonumber
\end{align}
where $x_{ij}$ is the flow going from node $i$ to node $j$, $\delta^+_i$ (resp. $\delta^-_i$) is the set of successor (resp. predecessor) nodes of node $i$ and $b_i \in \{-1,0,1\}$.

\begin{theorem}
BC and RC are equivalent for $(F)$.
\end{theorem}

\begin{proof}
By definition, RC implies BC for $(F)$. Conversely, assume BC for $(F)$. 

In order to show the converse, we show that if $k \in\llbracket l_{ij},  u_{ij}\rrbracket$, then $k$ belongs to a bound support for $(F)$. 

Let $i_0$ be a direct predecessor of $j_0$ in the graph. 
BC implies that there exists $x'$ (resp. $x''$) an integer solution of (F) such that $x'_{i_0j_0} = l_{i_0j_0}$ (resp. $x''_{i_0j_0} = u_{i_0j_0}$).
Let $\gamma \in [0,1]$ such as $k = \gamma \: l_{i_0j_0} + (1-\gamma) u_{i_0j_0} \in \mathbb{N}$. Let's show that there exists an integer solution of $(F)$ where $x_{i_0j_0} = k$.
We modify the bounds $\: \forall \:  (i,j) \in E$:
\begin{align*}
\widehat{l_{ij}} = \lfloor \gamma x'_{ij} + (1-&\gamma) x''_{ij} \rfloor  \\
\widehat{u_{ij}} = \lceil \gamma x'_{ij} + (1-&\gamma) x''_{ij} \rceil  
\end{align*}
We can then show that $\forall \: i,j \: \: \widehat{l_{ij}} \geq l_{ij}$ and $\widehat{u_{ij}} \leq u_{ij}$. The constraint matrix of $(F)$ is totally unimodular since it is a flow problem, hence the application of the simplex algorithm to $(F)$ with the updated bounds gives an integer solution $x'''$ where $x_{i_0j_0} = k$.
\end{proof}

As constraints \eqref{flotMILP1} are flow constraints and the setup variables $Y_t$ are binary, the following result follows.
\begin{corollary}
\label{BCPoly}
Achieving BC on $(L_r)$ is equivalent to achieving RC on $(L_r)$.
\end{corollary}


When there are no holes in the domains of $X$ and $I$, RC is equivalent to AC for $(L_r)$. Note that there may exist holes in lot-sizing problems when considering batching constraints for instance. 


\section{Filtering the \textsc{LotSizing}  constraint}
\label{filteringLS}

This section describes the filtering of the \textsc{LotSizing} constraint. It also gives some implementation details to improve the incrementality of the global constraint.

Algorithm \ref{alg:propagateLS} gives an overview of the main filtering steps of \textsc{LotSizing}. Each step refers to the corresponding section for detailed explanations. When all setup variables are instantiated the problem amounts to a minimum cost flow problem (lines 1-2). If not, the general case is as follows. Firstly, the problem is transformed by removing all lower bounds (line 4) as \textsc{LotSizing} is defined with lower bounds and these can increase during the search. Secondly production and inventory costs lower bounds are computed (lines 5-6). Thirdly, when the overall problem is of reasonable size (lines 7-11) the remaining filtering is performed using dynamic programming. If not, the WISP relaxation is used and filtering is performed via the WISP support (lines 13-17).

\begin{algorithm}[t!]
\caption{filtering algorithm of  \textsc{LotSizing}}
\label{alg:propagateLS}
\begin{algorithmic}[1]
\If {all the $Y$ are instantiated}
\State solve the min flow problem and instantiate $X$ and $I$ (\S \: \ref{dominance})
\Else
\State check feasibility (Corollary \: \ref{BCPoly})
\State remove lower bounds (\S \: \ref{transformation})
\State update $\underline{\mathit{Cp}}$ with flow relaxation restricted to production costs (\S \: \ref{costLB})
\State update $\underline{\mathit{Ch}}$ with flow relaxation restricted to inventory costs (\S \: \ref{costLB})
\If {the DP is scalable}
\State update $\underline{\mathit{C}}$ with DPLS (\S \: \ref{costLB})
\State filter variables via DP filtering (\S \: \ref{DPFiltering})
\State update $\underline{\mathit{Cs}}$ with DP (\S \: \ref{costLB})
\State filter variables via DP filtering (\S \: \ref{DPFiltering})
\Else 
\State compute all the $C_{uv}$ with appropriate relaxations (\S \: \ref{ExamplesLB})
\State update $\underline{\mathit{C}}$ with DPWisp  (\S \: \ref{WISP})
\State filter variables via WISP support filtering (\S \: \ref{scalingFiltering})
\State update $\underline{\mathit{Cs}}$ with DPWisp (\S \: \ref{wispCs})
\State filter variables via WISP support filtering (\S \: \ref{scalingFiltering})
\EndIf
\EndIf
\State \textbf{end algorithm}
\end{algorithmic}
\end{algorithm}

\subsection{Filtering when the setup variables are instantiated}
\label{dominance}

When all the setup variables $Y$ are instantiated, problem $(L)$ becomes polynomial and amounts to a minimum cost flow problem. Solving a minimum cost flow problem on the flow graph presented in Figure \ref{RLeqFLotFig} finds a solution to $(L)$ that minimizes $C$. This allows the user to branch only on the $Y$ variables since the solver can instantiate all the other variables in polynomial time when the setup variables are instantiated. Note that when using \textsc{Lotsizing} in a more complex model (with multiple \textsc{Lotisizing} or with additional constraints) the resulting problem, when all the $Y$ are instantiated, may not be polynomial. Therefore we let the user specify when stating the constraint if this property holds or not. Note also that a minimum cost flow dedicated filtering algorithm \cite{steiger2011efficient} can be used at this stage. Either of these two options allows the user to branch only on the $Y$ variables.

\subsection{Filtering cost lower bounds}
\label{costLB}
Lower bounds of the cost variables are computed as follows:
\begin{itemize}
	\item A lower bound on the production cost $\mathit{Cp}$ is computed by solving a minimum cost flow problem on the graph presented in \S \: \ref{RLeqFLot} considering only the variable production costs.
	\item A lower bound on the inventory cost $\mathit{Ch}$ is computed by solving a minimum cost flow problem on the graph presented in \S \: \ref{RLeqFLot} considering only the variable inventory costs.
	\item A lower bound on the global cost $\mathit{C}$ is computed using DPLS. The lower bound is given by the value $\textsl{f} \: (T, 0)$.
	\item A lower bound on the setup cost $\mathit{Cs}$ is computed via dynamic programming as well. We consider here the problem $(L)$ without production or inventory costs. As mentioned in the literature review, this problem can be solved using the traditional knapsack dynamic programming algorithm slightly adapted to take into account the inventory upper bounds (we call this algorithm DPKnap).
\end{itemize}

\subsection{Filtering X and I via dynamic programming}
\label{DPFiltering}
DPLS gives a lower bound of the global cost $C$. In order to filter the variables we use the tables created by DPLS and reverse DPLS. Remember that in the graph described in \ref{DP}, $\textsl{f} \: (t, I_t)$ can be seen as the shortest path from the node $(0,0)$ to $(t, I_t)$ and $\textsl{f}_{r} (t, I_t)$ is the shortest path from $(t, I_t)$ to $(T,0)$. We filter each value $i_t$ in the domain of $I_t$ in $O(T I_{max})$:
\begin{align}
&\forall \: t \in \llbracket 1, T \rrbracket, i_t \in D(I_{t}) \notag \\
&\label{filtDPI} \textsl{f}\:(t, i_t) + \textsl{f}_{r}(t, i_t) > \overline{C} \Rightarrow I_{t} \neq i_t
\end{align}
We filter each value in the domain of $X_t$ in $O(T I_{max}^2)$:
\begin{align}
&\forall \: t \in \llbracket 1, T \rrbracket, i_{t-1} \in D(I_{t-1}), i_t \in D(I_{t}) \notag \\
&\label{filtDPX} \textsl{f}\:(t-1, i_{t-1}) + cost(t, i_{t-1}, i_t) + \textsl{f}_{r}(t, i_t) > \overline{C} \Rightarrow X_{t} \neq x_t
\end{align}
where $x_t = d_{t} + i_t - i_{t-1}$ and  $cost(t, i_{t-1}, i_t) = \mathds{1}_{X_t > 0} s_{t} + p_{t} x_t + h_{t} i_t$.

\subsection{Scaling the filtering based on dynamic programming}
\label{scalingFiltering}
In the case that DPLS has memory issues on the overall problem $(L)$, we solve a WISP (see Section \ref{WISP}) to find a lower bound on $C$. We can adapt the filtering rules \eqref{filtDPI} and \eqref{filtDPX} on the sub-problems of reasonable size. In order to compare the shortest paths to the global upper bound $\overline{C}$, we need to have a lower bound on the cost of the production outside the sub-problem. We use DPWisp and its reverse version to do so. We can then define:
\begin{itemize}
	\item $ \textrm{lbBefore}(t) = \textsl{wisp}\:(\frac{t(t-1)}{2})$ which is the best bound we can get by combining the sub-problems ending by at most $t$. It is a lower bound on the satisfaction of the demands $d_1$ to $d_t$. We set $\textrm{lbBefore}(0) = 0$.
	\item $ \textrm{lbAfter}(t) = \textsl{wisp}_r(\frac{(T-t)(T- t + 1)}{2})$ which is the best bound we can get by combining the sub-problems starting at $t$. It is a lower bound of the cost for satisfying the demands from $d_t$ to $d_T$. We set $\textrm{lbAfter}(T+1) = 0$.
\end{itemize}

Figure \ref{WispFilering} represents the different lower bounds computed while filtering the value $i_t$ for the variable $I_t$. $I_t$ belongs to the sub-problem $(L_{u,v})$, $lbBefore(u-1)$ and $lbAfter(v+1)$ are computed via DPWisp outside that sub-problem. 

\begin{figure}[!h]
\centering
   \includegraphics[scale = 0.4]{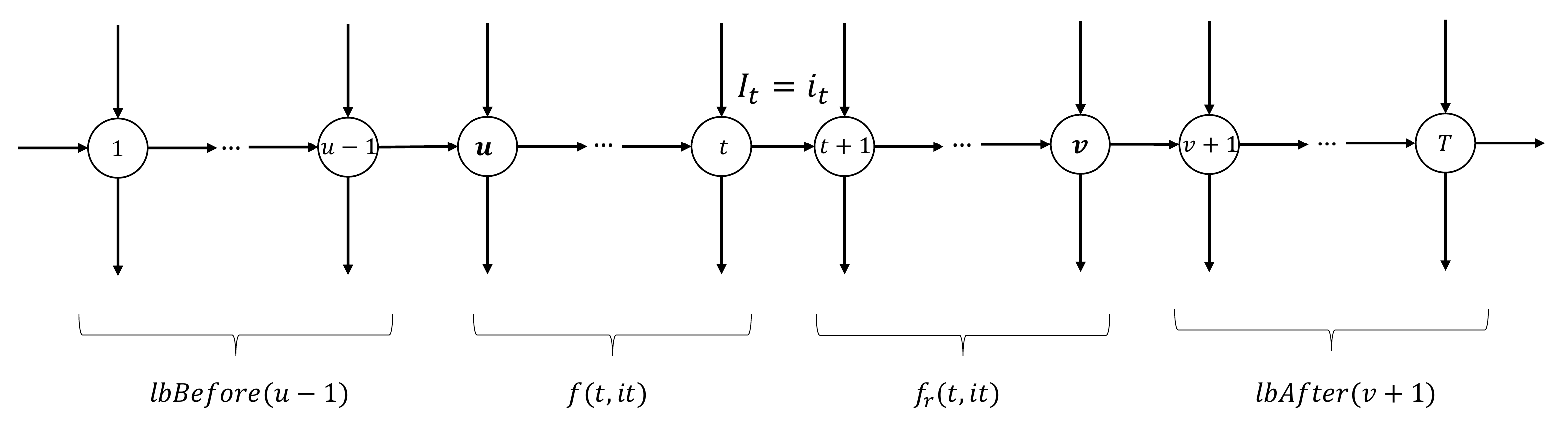}
   \caption{\label{WispFilering} Bounds when filtering $I_t$ with the WISP support filtering}
\end{figure}

We filter the variables via the two following rules:
\begin{align*}
&\forall \: t \in \llbracket u, v \rrbracket, i_t \in D(I_{t}) \\
&\textrm{lbBefore}(u-1) + \textsl{f}\:(t, i_t) + \textsl{f}_{r}(t, i_t) + \textrm{lbAfter}(v+1) > \overline{C} \Rightarrow I_{t} \neq i_t\\
\\
&\forall \: t \in \llbracket u, v \rrbracket, i_{t-1} \in D(I_{t-1}), i_t \in D(I_{t})  \\
&\textrm{lbBefore}(u-1) + \textsl{f}\:(t-1, i_{t-1}) + cost(t,i_{t-1}, i_t) + \textsl{f}_{r}(t, i_t) + \textrm{lbAfter}(v+1) > \overline{C}\\
& \Rightarrow X_{t} \neq x_t
\end{align*}

For all the sub-problems of the support (or solution) of the WISP, we can compute DPLS and its reverse version. Note that here $\textsl{f}\:(t, i_t)$ and $\textsl{f}_{r}(t, i_t)$ come from DPLS applied to the sub-problems. Hence when considering sub-problem $(L_{u,v})$ and $u\leq t\leq v$, $\textsl{f}\:(t, i_t)$ is a lower bound of the cost for satisfying demands $d_u$ to $d_t$ and having $I_t=i_t$. Similarly $\textsl{f}_{r}(t, i_t)$ is a lower bound on the cost for satisfying the demands $d_{t+1}$ to $d_v$ and having $I_t=i_t$. Moreover, since the sub-problems consider only the demands $d_u$ to $d_v$, we cannot filter values greater than or equal to $\sum_{k=t+1}^{v}{d_k}$ for $I_t$ and $X_t$. Indeed, greater values might be used to satisfy demands outside the sub-problem and are not considered when computing DPLS.

Note that although the scaling may affect the quality of the filtering, it is a pragmatic rule applied to avoid wakening costly propagation when little filtering is expected.


\subsection{Adaptation to take into account the setup cost}
In order to adapt the filtering to take into account $\overline{\mathit{Cs}}$, we do not take into account the production and inventory costs. The resulting problem can then be tackled by dynamic programming with DPKnap. The variables can be filtered via the WISP support with the DPKnap computed on the sub-problems. The filtering rules are applied with the upper bound of $\mathit{Cs}$.

\section{Numerical results on the single-item lot-sizing problem}
\label{numResultsLS}
This section validates our global constraint and the filtering mechanisms described above. We compare the performance of \textsc{LotSizing} to four other methods on the single-item lot-sizing problem. 

\paragraph{Five methods}
The five methods that solve the single-item lot-sizing are:
\begin{itemize}
\item A basic CP model (CP\_Basic), which is a decomposition of the single-item lot-sizing problem basically equivalent to the MILP model with the implication constraints $X_t > 0 \Rightarrow Y_t = 1, \forall \: t \in\llbracket 1, T \rrbracket$ instead of the setup constraints \eqref{setupMILP1}
\item A CP model with our \textsc{LotSizing} global constraint (CP\_LS)
\item The dynamic programming algorithm presented in \ref{DP} (DP)
\item The classical aggregated MILP model (MILP\_AGG)
\item The facility location MILP model (MILP\_UFL)
\end{itemize}

The MILP models were implemented with CPLEX version 12.6 and the CP models in Choco 3.3 \cite{choco}.

\paragraph{Branching heuristics for the CP models}
A default branching heuristic is used to instantiate the variables to their lower bounds in a lexicographic order (chronological order here).
The property described in \ref{dominance} is valid for the single-item lot-sizing problem. For the sake of comparison, the same improvement is done for CP\_Basic. The search space is thus restricted to the setup variables ($Y$) for both CP models.

\paragraph{Cost upper bound}
Since we want to assess the quality of the filtering and \textsc{LotSizing} uses cost-based filtering, we choose to have the best possible upper bound on the global cost at the start of the resolution: the optimal cost. This means that our models still have to find the optimal solution. By setting the initial upper bound to the optimal value, we simply aim to avoid the issue of finding a good enough initial solution that would "activate the filtering" we want to assess. It gives a simple and identical set-up to all compared approaches and allows us to focus the analysis on the filtering we have been investigating. In a more realistic setting, the model CP\_LS can be used to find upper bounds.

\paragraph{Instance parameters}
The single-item instances are generated based on the parameters $d_{\text{avg}}$, $e$, $\delta$, $\theta$, $\lambda$, and $T$ as follows:
\begin{itemize}
\item The inventory costs are constant and equal to $1$ (i.e. $h_t=h=1$).
\item The setup and the unitary production costs are generated using two parameters: $e$ and $\theta$.
\begin{itemize}
\item $e$ represents the overall unitary production cost (i.e. the unitary production cost if the production capacity is saturated: $p_t \: \overline{\alpha_t} + s_t$ divided by $\overline{\alpha_t}$). We set $e=10$.
\item $\theta \in [0,1]$ represents the portion of the setup cost to the unitary production cost. The overall production cost at $t$ (i.e. $e \: \overline{\alpha_t}$) will be imputable for $\theta$ to its setup cost and for $1-\theta$ to the unitary production cost at $t$. For each period, $\theta$ is uniformly randomly set in the interval $[0,1]$.
\end{itemize}
\item The demand is uniformly randomly generated in the interval $[d_{\text{avg}}-\delta, d_{\text{avg}} + \delta]$.
\item The production and inventory capacities are constant and equal to $\lambda \: d_{\text{avg}}$. 
\end{itemize}
For each problem, we give the set of parameters that were used to generate the instances. Each class of instances contains 10 instances.

\paragraph{Experimentation setup}
All the tests are run under Windows 8 on an Intel Core i5 @ 2.5 GHz with 12GB of RAM. We set a time limit of 200s and a memory limit of 4GB of RAM. The indicator NODE is the average number of nodes computed by each model on the class. CPU corresponds to the average CPU time used by the models. RNB is the average gap of the root node lower bound to the optimal value. LR is the average gap of the linear relaxation to optimal. Finally OPT is the number of solved instances in the class. The means are computed over all the instances of each class.

\subsection{Single-item lot-sizing}
The five instance classes are:
\begin{itemize}
	\item C1LS : $d_{\text{avg}} = 1000$, $\delta = 100$, $\theta \in [0.8,1]$, $\lambda = 3$, $T=40$
	\item C2LS : $d_{\text{avg}} = 1000$, $\delta = 500$, $\theta \in [0.4,0.6]$, $\lambda = 3$,$T=40$
	\item C3LS : $d_{\text{avg}} = 1000$, $\delta = 100$, $\theta  \in [0.8,1]$, $\lambda = 3$, $T=80$
	\item C4LS : $d_{\text{avg}} = 1000$, $\delta = 500$, $\theta \in [0.4,0.6]$, $\lambda = 3$, $T=80$
	\item C5LS : $d_{\text{avg}} = 1000$, $\delta = 50$, $\theta = 0.5$, $\lambda = 3$, $T=40$
\end{itemize}

The results are presented in Tables \ref{lsCPDP} and \ref{lsMILP}.

\begin{table}[!h]
\footnotesize
\begin{center}
\begin{tabular}{|l|c|c|c|c|c|c|c|c|c|c|}
\hline
&  \multicolumn{4}{c|}{CP\_Basic} &  \multicolumn{4}{c|}{CP\_LS} & \multicolumn{2}{c|}{DP} \\
\hline
 Class  & NODE & CPU & RNB & OPT & NODE & CPU & RNB & OPT & CPU & OPT \\
\hline
C1LS	&	8.3E+06	&	200	&	100\%	&	0	&	\textbf{1}	&	1.1	&	\textbf{0}\%	&	10	&	0.2	&	10	\\
\hline																					
C2LS	&	1.7E+06	&	200	&	65\%	&	0	&	\textbf{1}	&	1.0	&	\textbf{0}\%	&	10	&	\textbf{0.2}	&	10	\\
\hline																					
C3LS	&	6.5E+06	&	200	&	100\%	&	0	&	\textbf{28}	&	2.1	&	\textbf{0}\%	&	10	&	\textbf{0.4}	&	10	\\
\hline																					
C4LS	&	3.4E+05	&	200	&	64\%	&	0	&	\textbf{35}	&	2.2	&	\textbf{0}\%	&	10	&	\textbf{0.4}	&	10	\\
\hline																					
C5LS	&	8.5E+06	&	200	&	56\%	&	0	&	\textbf{1}	&	1.1	&	\textbf{0}\%	&	10	&	\textbf{0.2}	&	10	\\
\hline																					

\end{tabular}
\normalsize
\caption{\label{lsCPDP} Single-item lot-sizing - CP and DP}
\end{center}
\end{table}

\begin{table}[!h]
\footnotesize
\begin{center}
\begin{tabular}{|l|c|c|c|c|c|c|c|c|c|c|}
\hline
&  \multicolumn{5}{c|}{MILP\_AGG}& \multicolumn{5}{c|}{MILP\_UFL} \\
\hline
 Class  & NODE & CPU & RNB & LR & OPT & NODE & CPU & RNB & LR & OPT  \\
\hline

C1LS	&	580	&	\textbf{0.1}	&	1\%	&	10\%	&	10	&	460	&	0.6	&	1\%	&	\textbf{3}\%	&	10	\\
\hline																					
C2LS	&	1360	&	\textbf{0.2}	&	2\%	&	10\%	&	10	&	1643	&	1.4	&	2\%	&	\textbf{3}\%	&	10	\\
\hline																					
C3LS	&	3213	&	1.7	&	2\%	&	11\%	&	10	&	13109	&	54.1	&	3\%	&	\textbf{3}\%	&	10	\\
\hline																					
C4LS	&	2222	&	1.6	&	1\%	&	10\%	&	10	&	14366	&	61.6	&	2\%	&	\textbf{2}\%	&	10	\\
\hline																					
C5LS	&	1691	&	0.3	&	2\%	&	11\%	&	10	&	9336	&	4.5	&	2\%	&	\textbf{3}\%	&	10	\\
\hline

\end{tabular}
\normalsize
\caption{\label{lsMILP} Single-item lot-sizing - MILP}
\end{center}
\end{table}

These tables show that:
\begin{itemize}
\item As expected, the basic CP model has a very large search space as it does not propagate any strong reasoning. We therefore did not use CP\_Basic for the following results.
\item As the upper bound provided is optimal and there is no upper bound on $\mathit{Cp}$,$\mathit{Ch}$ and $\mathit{Cs}$, CP\_LS achieves AC at the root node and branches backtrack free towards an optimal solution. 
\item The linear relaxation of MILP\_AGG is not as good as the MILP\_UFL's as it was expected due to the setup constraints \eqref{setupMILP1}. CPLEX however provides a better root node lower bound for MILP\_AGG. MILP\_UFL is not as competitive as MILP\_AGG because of the number of variables and constraints.
We therefore did not use MILP\_UFL for the following results.
\end{itemize}

\subsection{Scaling the global constraint}
We then test the WISP support filtering described in \ref{scalingFiltering} when the DP has memory issues. In order to generate memory issues for the DP, we add high consumption peaks in the instances. The peaks are added in periods 6 to 9, 12 to 15, 22 to 25 and 32 to 36 and correspond to demands of 50,000. When computing the global lower bound with the WISP, no sub-problem containing a demand peak is solved via dynamic programming since the peaks increase the complexity of the DP. The lower bound on these sub-problems is therefore their linear relaxation. 
The five instance classes have the following parameters:
\begin{itemize}
	\item C1Peaks : $d_{\text{avg}} = 100$, $\delta = 50$, $\theta \in [0.8,1]$, $\lambda = 4$, $T=40$
	\item C2Peaks : $d_{\text{avg}} = 100$, $\delta = 50$, $\theta \in [0.4,0.6]$, $\lambda = 4$,$T=40$
	\item C3Peaks : $d_{\text{avg}} = 100$, $\delta = 50$, $\theta  = 0.5$, $\lambda = 4$, $T=40$
	\item C4Peaks : $d_{\text{avg}} = 100$, $\delta = 20$, $\theta \in [0.8,1]$, $\lambda = 4$, $T=40$
	\item C5Peaks : $d_{\text{avg}} = 100$, $\delta = 20$, $\theta \in [0.4,0.6]$, $\lambda = 4$, $T=40$
\end{itemize}
The branching heuristic is adapted to select first the setup variables of the high demand periods. Table \ref{lsScaling} compares the three models CP\_LS, MILP\_AGG and DP on these big instances.
\begin{table}[!h]
\footnotesize
\begin{center}
\begin{tabular}{|l|c|c|c|c|c|c|c|c|c|c|c|}
\hline
&  \multicolumn{4}{c|}{CP\_LS} &  \multicolumn{5}{c|}{MILP\_AGG} & \multicolumn{2}{c|}{DP} \\
\hline
 Class  & NODE & CPU & RNB & OPT & NODE & CPU & RNB & LR & OPT & CPU & OPT \\
\hline
C1Peaks	&	125	&	1.1	&	2\%	&	10	&	\textbf{0}	&	\textbf{0.0}	&	\textbf{0}\%	&	13\%	&	10	&	23.8	&	10	\\
\hline																							
C2Peaks	&	468	&	5.7	&	2\%	&	10	&	\textbf{0}	&	\textbf{0.0}	&	\textbf{0}\%	&	6\%	&	10	&	23.4	&	10	\\
\hline																							
C3Peaks	&	2784	&	21.0	&	3\%	&	10	&	\textbf{1}	&	\textbf{0.0}	&	\textbf{0}\%	&	7\%	&	10	&	22.6	&	10	\\
\hline																							
C4Peaks	&	408	&	3.7	&	3\%	&	10	&	\textbf{0}	&	\textbf{0.0}	&	\textbf{0}\%	&	15\%	&	10	&	23.4	&	10	\\
\hline																							
C5Peaks	&	446	&	5.5	&	2\%	&	10	&	\textbf{0}	&	\textbf{0.0}	&	\textbf{0}\%	&	7\%	&	10	&	24.0	&	10	\\
\hline																					
\end{tabular}
\normalsize
\caption{\label{lsScaling} Scaling the global constraint}
\end{center}
\end{table}

This table shows that:
\begin{itemize}
\item The filtering is lighter, hence the root node lower bound gap increases as well as the number of nodes.
\item The resolution is however faster than the DP.
\item Although the linear relaxation degrades, CPLEX pre-processing behaves very well as shown by the root node lower bound.
\end{itemize} 

\section{Single-item lot-sizing with side constraints}
\label{numResultsAC}
We consider the single-item lot-sizing problem $(L)$ with three side constraints (domain disjunction, limited production rate and a combination of the two). The instances created here are generated the same way as before and we added leveled production and/or constrained production rate. For the following tests, we compared only CP\_LS to MILP\_AGG and to DP when it is relevant.

\subsection{Disjunctive production constraints}
We consider here that the production is leveled. The domains of each variable $X_t$ is defined by a disjunction of $n_t$ integer intervals $ K_k = \llbracket \underline{K_k}, \overline{K_k} \rrbracket, \: \forall \: k \in \llbracket 1 , n_t \rrbracket$:
$$ D(X_t) = \{0\} \cup K_1  \cup \ldots \cup K_{n_t}$$

DPLS can take into account the disjunctions without any loss of complexity. We add the following constraints to the MILP model MILP\_AGG:
\begin{align}
X_t &= \sum_{k=1}^{n_t} {X_t^k} &&\forall \: t = 1 \ldots T \label{prodMILPDisj}\\
Y_t &= \sum_{k=1}^{n_t} {Y_t^k} &&\forall \: t = 1 \ldots T \label{setupMILPDisj0}\\
X_t^k &\leq Y_t^k \overline{K_k} &&\forall \: t = 1 \ldots T, k = 1 \ldots n_t\label{setupMILPDisj1}\\
Y_t^k \underline{K_k} & \leq X_t^k &&\forall \: t = 1 \ldots T, k = 1 \ldots n_t\label{setupMILPDisj2}
\end{align}

The ten classes for this problem are:
\begin{itemize}
	\item C1Disj : $d_{\text{avg}} = 100$, $\delta = 50$, $\theta \in [0.8,1]$, $\lambda = 5$, $T=40$
	\item C2Disj : $d_{\text{avg}} = 100$, $\delta = 60$, $\theta \in [0.4,0.6]$, $\lambda = 5$,$T=40$
	\item C3Disj : $d_{\text{avg}} = 100$, $\delta = 70$, $\theta  \in [0.3,0.8]$, $\lambda = 5$, $T=40$
	\item C4Disj : $d_{\text{avg}} = 100$, $\delta = 30$, $\theta \in [0.6,1]$, $\lambda = 5$, $T=40$
	\item C5Disj : $d_{\text{avg}} = 100$, $\delta = 50$, $\theta \in [0.9,1]$, $\lambda = 5$, $T=40$
\end{itemize}
We generated C6Disj, C7Disj, C8Disj, C9Disj and C10Disj that have the same parameters than the five instances above, but with $T=80$. The disjunctions are added as follows: $D(X_t) = \llbracket 0, 30\rrbracket \cup \llbracket 100, 150\rrbracket \cup \llbracket 200, 240\rrbracket$. Table \ref{lsDisj} gives the numerical results for the single-item lot-sizing with disjunctions.

\begin{table}[!h]
\footnotesize
\begin{center}
\begin{tabular}{|l|c|c|c|c|c|c|c|c|c|c|c|}
\hline
&  \multicolumn{4}{c|}{CP\_LS} &  \multicolumn{5}{c|}{MILP\_AGG} & \multicolumn{2}{c|}{DP} \\
\hline
 Class  & NODE & CPU & RNB & OPT & NODE & CPU & RNB & LR & OPT & CPU & OPT \\
\hline
C1Disj	&	\textbf{1}	&	\textbf{0.0}	&	\textbf{0}\%	&	10	&	6.4E+05	&	162.1	&	43\%	&	52\%	&	2	&	\textbf{0.0}	&	10	\\
\hline																							
C2Disj	&	\textbf{2}	&	\textbf{0.0}	&	\textbf{0}\%	&	10	&	1.4E+04	&	6.0	&	27\%	&	38\%	&	10	&	\textbf{0.0}	&	10	\\
\hline																							
C3Disj	&	\textbf{2}	&	\textbf{0.0}	&	\textbf{0}\%	&	10	&	1.6E+03	&	0.5	&	27\%	&	38\%	&	10	&	\textbf{0.0}	&	10	\\
\hline																							
C4Disj	&	\textbf{2}	&	\textbf{0.0}	&	\textbf{0}\%	&	10	&	2.2E+04	&	6.3	&	39\%	&	48\%	&	10	&	\textbf{0.0}	&	10	\\
\hline																							
C5Disj	&	\textbf{1}	&	\textbf{0.0}	&	\textbf{0}\%	&	10	&	9.6E+05	&	200.0	&	52\%	&	61\%	&	0	&	\textbf{0.0}	&	10	\\
\hline																							
C6Disj	&	\textbf{2}	&	0.1	&	\textbf{0}\%	&	10	&	3.2E+05	&	200.0	&	43\%	&	52\%	&	0	&	\textbf{0.0}	&	10	\\
\hline																							
C7Disj	&	\textbf{2}	&	0.1	&	\textbf{0}\%	&	10	&	5.0E+04	&	38.3	&	27\%	&	39\%	&	10	&	\textbf{0.0}	&	10	\\
\hline																							
C8Disj	&	\textbf{1}	&	0.1	&	\textbf{0}\%	&	10	&	5.7E+03	&	4.5	&	28\%	&	39\%	&	10	&	\textbf{0.0}	&	10	\\
\hline																							
C9Disj	&	\textbf{2}	&	0.1	&	\textbf{0}\%	&	10	&	3.6E+04	&	21.9	&	38\%	&	48\%	&	10	&	\textbf{0.0}	&	10	\\
\hline																							
C10Disj	&	\textbf{2}	&	0.1	&	\textbf{0}\%	&	10	&	3.3E+05	&	200.0	&	52\%	&	60\%	&	0	&	\textbf{0.0}	&	10	\\
\hline																							
\end{tabular}
\normalsize
\caption{\label{lsDisj} Single-item lot-sizing with disjunctions}
\end{center}
\end{table}
Table \ref{lsDisj} shows that the CP and DP models are very fast to solve these instances. The property described in \ref{dominance} concerning the setup variables is not valid for this problem: indeed the flow with disjunctions is not polynomial. However as \textsc{LotSizing}'s filtering uses the DP, the global constraint can handle disjunctions on the domains of the production variables. Therefore CP\_LS achieves AC at the root node and branches backtrack free towards an optimal solution. Unsurprisingly we note that the MILP model does not handle these disjunction constraints well.

\subsection{Q/R constraints}
Q/R constraints are interesting side constraints for single-item lot-sizing problems \cite{hellion2014,hellion2015}. They relate to the production rate and state that, given two integers $Q$ and $R$, there must be at least $Q$ and at most $R$ periods between two consecutive productions.
Dynamic programming rapidly gets memory issues here, as the states should take into account what happened at least $R$ periods before. The Q/R constraints can be modeled by two \textsc{Sequence} constraints stated as follows:
\begin{align*}
&\textsc{Sequence}(0,1,Q+1, [ Y_1, \ldots, Y_T ] , \{1\})\\
&\textsc{Sequence}(1, R+1 ,R+1, [ Y_1, \ldots, Y_T ] , \{1\})
\end{align*}
The \textsc{Sequence} constraint is defined as follows \cite{beldiceanu2001}:
$\textsc{Sequence}(l,u,k, [ Z_1, \ldots, Z_n ] , v)$ holds if and only if:
 $$\forall \: 1 \leq i \leq n-k+1 \: \: \: \: l \leq | \{i \: | \: Z_i \in v\}| \leq u$$
We add the following constraints to the model MILP\_AGG:
\begin{align}
\sum_{t=u}^{v} {Y_t} & \leq 1 && \forall \: u,v \in \llbracket 1, T \rrbracket \: \text{s.t.} \: v-u+1 = Q+1 \label{setupMILPQR0}\\
\sum_{t=u}^{v} {Y_t} & \geq 1 && \forall \: u,v \in \llbracket 1, T \rrbracket \: \text{s.t.} \: v-u+1 = R+1 \label{setupMILPQR1}
\end{align}
We add the $\#_{uv}$ variables that count the number of effective production periods between period $u$ and period $v$ included:
\begin{align}
\#_{uv} & = \sum_{t=u}^{v} {Y_t} && \forall \: u,v \in \llbracket 1, T \rrbracket  \label{carduvCP1}
\end{align}
These variables enable us to use the encoding of \textsc{Sequence} presented in \cite{brand2007} to propagate the Q/R constraints. We also add the useful following redundant constraints:
\begin{align}
\#_{1t} + \#_{t+1 T} &= \#_{1T} && \forall \: t \in \llbracket 2, T-1 \rrbracket \label{carduvCP3} \\
\#_{1t} + Y_{t+1} &= \#_{1 t+1} && \forall \: t \in \llbracket 2, T-1 \rrbracket \label{carduvCP4}
\end{align}

The ten classes (C1QR, $\ldots$, C10QR) for this problem have the same parameters than C1Disj, $\ldots$, C10Disj to which we add (Q=2, R=6) for classes 1, 2, 3, 6, 7, 8 and (Q=3, R=7) for classes 4, 5, 9, 10.
Table \ref{lsQR} compares CP\_LS to MILP\_AGG on the instances with Q/R.

\begin{table}[!h]
\footnotesize
\begin{center}
\begin{tabular}{|l|c|c|c|c|c|c|c|c|c|}
\hline
&  \multicolumn{4}{c|}{CP\_LS} &  \multicolumn{5}{c|}{MILP\_AGG}\\
\hline
 Class  & NODE & CPU & RNB & OPT & NODE & CPU & RNB & LR & OPT \\
\hline
C1QR	&	\textbf{2}	&	0.3	&	\textbf{0}\%	&	10	&	17	&	0.0	&	1\%	&	14\%	&	10	\\
\hline																			
C2QR	&	\textbf{22}	&	0.2	&	\textbf{1}\%	&	10	&	56	&	0.1	&	\textbf{1}\%	&	13\%	&	10	\\
\hline																			
C3QR	&	84	&	0.4	&	\textbf{1}\%	&	10	&	\textbf{27}	&	0.0	&	\textbf{1}\%	&	12\%	&	10	\\
\hline																			
C4QR	&	\textbf{1}	&	0.3	&	\textbf{0}\%	&	10	&	20	&	0.0	&	1\%	&	12\%	&	10	\\
\hline																			
C5QR	&	\textbf{1}	&	0.4	&	\textbf{0}\%	&	10	&	13	&	0.0	&	1\%	&	16\%	&	10	\\
\hline																			
C6QR	&	\textbf{772}	&	7.2	&	\textbf{1}\%	&	10	&	943	&	0.5	&	2\%	&	16\%	&	10	\\
\hline																			
C7QR	&	6488	&	42.8	&	\textbf{0}\%	&	10	&	\textbf{601}	&	0.4	&	1\%	&	14\%	&	10	\\
\hline																			
C8QR	&	26716	&	134.1	&	\textbf{1}\%	&	5	&	\textbf{392}	&	0.3	&	\textbf{1}\%	&	13\%	&	10	\\
\hline																			
C9QR	&	\textbf{1}	&	1.1	&	\textbf{0}\%	&	10	&	510	&	0.3	&	1\%	&	15\%	&	10	\\
\hline																			
C10QR	&	\textbf{21}	&	1.9	&	\textbf{0}\%	&	10	&	1175	&	0.5	&	3\%	&	18\%	&	10	\\
\hline

\end{tabular}
\normalsize
\caption{\label{lsQR} Single-item lot-sizing with Q/R}
\end{center}
\end{table}

The linear relaxation and root node lower bound of MILP\_AGG has slightly worsened without degrading the performance of the model. CP\_LS stays competitive on most of the instances.

\subsection{Disjunctive with Q/R constraints}
We add both Q/R and disjunctive production constraints. The problem cannot be tackled via DP due to the Q/R constraints, hence we compared CP\_LS to MILP\_AGG. The instances have the same parameters than before with both the disjunctions and the Q/R parameters presented for the latter problems. The results are shown in table \ref{lsQRDisj}.

\begin{table}[!h]
\footnotesize
\begin{center}
\begin{tabular}{|l|c|c|c|c|c|c|c|c|c|}
\hline
&  \multicolumn{4}{c|}{CP\_LS} &  \multicolumn{5}{c|}{MILP\_AGG}\\
\hline
 Class  & NODE & CPU & RNB & OPT & NODE & CPU & RNB & LR & OPT \\
\hline
C1DijsQR	&	\textbf{1}	&	\textbf{0.3}	&	\textbf{0}\%	&	10	&	1808	&	0.8	&	7\%	&	19\%	&	10	\\
\hline																			
C2DijsQR	&	\textbf{6}	&	\textbf{0.2}	&	\textbf{0}\%	&	10	&	637	&	0.4	&	2\%	&	14\%	&	10	\\
\hline																			
C3DijsQR	&	\textbf{67}	&	\textbf{0.4}	&	\textbf{1}\%	&	10	&	710	&	\textbf{0.4}	&	2\%	&	13\%	&	10	\\
\hline																			
C4DijsQR	&	\textbf{3}	&	0.4	&	\textbf{0}\%	&	10	&	964	&	\textbf{0.3}	&	5\%	&	16\%	&	10	\\
\hline																			
C5DijsQR	&	\textbf{30}	&	0.5	&	\textbf{1}\%	&	10	&	209	&	\textbf{0.2}	&	12\%	&	25\%	&	10	\\
\hline																			
C6DijsQR	&	\textbf{669}	&	\textbf{4.7}	&	\textbf{0}\%	&	10	&	40150	&	53.4	&	5\%	&	19\%	&	10	\\
\hline																			
C7DijsQR	&	\textbf{3471}	&	17.8	&	\textbf{0}\%	&	10	&	4839	&	\textbf{9.2}	&	2\%	&	15\%	&	10	\\
\hline																			
C8DijsQR	&	22386	&	94.0	&	\textbf{1}\%	&	7	&	\textbf{4618}	&	\textbf{7.4}	&	3\%	&	14\%	&	10	\\
\hline																			
C9DijsQR	&	\textbf{53}	&	\textbf{1.5}	&	\textbf{0}\%	&	10	&	5066	&	8.8	&	4\%	&	17\%	&	10	\\
\hline																			
C10DijsQR	&	\textbf{7}	&	\textbf{2.6}	&	\textbf{1}\%	&	10	&	1663	&	1.7	&	11\%	&	26\%	&	10	\\
\hline
\end{tabular}
\normalsize
\caption{\label{lsQRDisj} Single-item lot-sizing with disjunctives and Q/R}
\end{center}
\end{table}
On some classes, CP\_LS does not solve all the instances yet is competitive compared to MILP\_AGG and has a near optimal root node lower bound.\\

We summarize below the main conclusions of our numerical study that hold for the set of instances under consideration.

For the single-item problem:
\begin{itemize}
\item CP\_Basic fails to find the optimal solution in a reasonable time due to the size of the search space and the lack of pruning. As no information is given on the costs, its root node lower bound is very far from the optimal value.
\item CP\_LS achieves arc consistency at the root node since it is given the optimal upper bound and is a competitive approach to find the optimal solution. It still has to branch when there exist multiple optimal solutions (see for instance classes C3LS and C4LS). CP\_LS can outscale DP on instances with demand peaks. The AC is not achieved at the root node since the scaling is based on a relaxation of the filtering mechanisms.
\item DP and  MILP\_AGG are the most competitive approaches.
\item MILP\_UFL finds the best RNB of the two MILP models but is slower overall.
\end{itemize}

For the single-item problem with additional constraints:
\begin{itemize}
\item Disjunctive constraints: CP\_LS and DP outperform MILP which does not handle well combinatorial constraints.
\item Q/R constraints: CP\_LS stays competitive on most instances while DP is not a suitable approach.
\item Q/R and Disjunctive constraints: CP\_LS is the fastest on most instance classes. As for C7 and C8 classes, note that MILP\_AGG was not very  troubled by the disjunctive constraints (Table 5) nor by the Q/R constraints (Table 6) unlike CP\_LS. Hence the difficulty of CP\_LS to rapidly solve these instances (Table 7).
\end{itemize}

\section{Conclusion}
In this paper, we defined a global constraint \textsc{LotSizing} for a capacitated single-item lot-sizing problem. 
Firstly, we presented a new  lower bound for this problem, based on a new decomposition of the problem into sub-problems. Secondly, we formally introduced our constraint and gave some complexity results. Thirdly, we developed  filtering rules for the \textsc{LotSizing} global constraint based on dynamic programming. Finally, we presented a proof of concept for the filtering of the constraint via several numerical results. We can conclude that our approach based on constraint programming  can yield interesting and competitive results for lot-sizing problems with side constraints. We however want to point out the limits of our numerical study that lie in the small size and variability of the set of instances that we chose for these first tests of the \textsc{LotSizing} global constraint.\\

The next step of this work will be to use the \textsc{LotSizing} global constraint as a building block to tackle multi-item and multi-echelon  problems. In multi-item, each item can be modeled as a \textsc{LotSizing} constraint and each of them infers on the feasibility and costs of how to produce its item. They however share some variables if for instance we consider shared setup costs, the setup variables would be shared among the \textsc{LotSizing} constraints, communicating information. In multi-level problems, each retailer can be modeled as a \textsc{LotSizing} constraint. The global constraints are linked since the input of a retailer (production variables) is the output of the previous retailer (demands). The constraint programming framework built around the \textsc{LotSizing} global constraint might very well benefit from several work that has been done on the relaxation of these types of problems \cite{van2014relaxations,zhang2012polyhedral}.

\bibliographystyle{abbrv}
\bibliography{biblio}

\end{document}